\documentclass[12pt]{amsart}
\usepackage{amsmath,dsfont,amsfonts,amssymb,amsxtra,latexsym,amscd,enumerate,amsthm}

   \def\to{\rightarrow}
 \def\ls{\lesssim} \def\gs{\gtrsim}
\def\la{\langle} \def\ra{\rangle}

\def\R{\mathbb{R}} \def\C{\mathbb{C}} \def\Z{\mathbb{Z}}
\def\N{\mathbb{N}} \def\S{\mathbb{S}}

\newcommand{\eps}{\epsilon} \DeclareMathOperator{\dist}{d}
\DeclareMathOperator{\med}{med}
\DeclareMathOperator{\supp}{supp}

\def\beq{\begin{equation}} \def\eeq{\end{equation}}
 
\def\beq{\begin{equation}} \def\eeq{\end{equation}}

\newtheorem{thm}{Theorem} 
\newtheorem{pro}[thm]{Proposition} \newtheorem{lem}[thm]{Lemma}

\theoremstyle{remark} \newtheorem{rmk}[thm]{Remark}
\theoremstyle{definition} 

\numberwithin{equation}{section} \numberwithin{thm}{section}

\begin{document}
\title[GWP and scattering for the Dirac-Klein-Gordon system]{On global
  well-posedness and scattering for the massive Dirac-Klein-Gordon system}

\author[I.~Bejenaru]{Ioan Bejenaru} \address[I.~Bejenaru]{Department
  of Mathematics, University of California, San Diego, La Jolla, CA
  92093-0112 USA} \email{ibejenaru@math.ucsd.edu}

\author[S.~Herr]{Sebastian Herr} \address[S.~Herr]{Fakult\"at f\"ur
  Mathematik, Universit\"at Bielefeld, Postfach 10 01 31, 33501
  Bielefeld, Germany} \email{herr@math.uni-bielefeld.de}

\begin{abstract}
  We prove global well-posedness and scattering for the massive
  Dirac-Klein-Gordon system with small initial data of subcritical
  regularity in dimension three. To achieve this, we impose a non-resonance condition on the masses.
\end{abstract}

\maketitle

\section{Introduction}\label{sect:intro}

The Dirac-Klein-Gordon system is a basic model of proton-proton
interactions (one proton is scattered in a meson field produced by a
second proton) or neutron-neutron interaction, see Bjorken and Drell
\cite{BjDr}.  In physics these are known as the strong interactions
which are responsible for the forces which bind nuclei.
  
The mathematical formulation of the Dirac-Klein-Gordon system is as
follows, see e.g.\ \cite{DaFoSe}:
\begin{equation} \label{DKG} \left\{
    \begin{aligned} 
      & (-i \gamma^\mu \partial_\mu + M ) \psi= \phi \psi   \\
      & (\Box + m^2) \phi = \psi^\dag \gamma^0 \psi
    \end{aligned}
  \right.
\end{equation}
Here, $\Box$ denotes the d'Alembertian $\Box=\partial_t^2-\Delta_x$,
$\psi: \R^{1+3} \rightarrow \C^4$ is the spinor field (column vector),
and $\phi: \R^{1+3}\rightarrow \R$ is a scalar field. For
$\mu=0,\ldots,3$, $\gamma^\mu$ are the $4 \times 4$ Dirac matrices
given by
\[
\gamma^0= \left( \begin{array}{cc} I_2 & 0 \\ 0 & -I_2 \end{array}
\right) , \qquad \gamma^j=\left( \begin{array}{cc} 0 & \sigma^j \\
    -\sigma^j & 0 \end{array} \right)
\]
where for $j=1,2,3$ the Pauli matrices $\sigma^j$ are
\[
\sigma^1= \left( \begin{array}{cc} 0 & 1 \\ 1 & 0 \end{array} \right)
, \qquad \sigma^2=\left( \begin{array}{cc} 0 & -i \\ i & 0 \end{array}
\right) , \qquad \sigma^3=\left( \begin{array}{cc} 1 & 0 \\ 0 &
    -1 \end{array} \right).
\]
$\psi^\dag$ denotes the conjugate transpose of $\psi$, i.e.\
$\psi^\dag=\overline{\psi}^t$. The matrices $\gamma^\mu$ satisfy the
following properties
\[
\gamma^\alpha \gamma^\beta + \gamma^\beta \gamma^\alpha = 2 g^{\alpha
  \beta} I_4, \qquad g^{\alpha \beta}= \mathrm{diag}(1,-1,-1,-1).
\]
We will study the Cauchy problem with initial condition
\begin{equation}\label{eq:i-cond}(\psi,\phi,\partial_t
  \phi)|_{t=0}=(\psi_0, \phi_0, \phi_1).
\end{equation}

Before turning to the mathematical analysis of the Dirac-Klein-Gordon
Equations we highlight a key property of the physical model presented
in Bjorken and Drell \cite[Chapter 10.2]{BjDr}. The mass $M$ is
effectively $938 \frac{MeV}{c^2}$ (proton) or $939 \frac{MeV}{c^2}$
(neutron). There are many types of meson fields, but those believed to
be major contributors to the nuclear force at large distances are the
$\pi$-mesons (pions) and their masses are $m=140 \frac{MeV}{c^2}$ for
$\pi^\pm$, $m=135 \frac{MeV}{c^2}$ for $\pi^0$. Heavier mesons such as
the K mesons (kaons) may also play a role for small impact parameter
collisions; the masses of a kaons are $m=494 \frac{MeV}{c^2}$ for
$K^\pm$ and $m=498 \frac{MeV}{c^2}$ for $K^0$.  It is then reasonable
to assume that in the Dirac-Klein-Gordon Equations it holds
\[
2M > m > 0.
\]
We are not implying that all mesons are lighter than baryons (protons
or neutrons in our context), but that this is a reasonable assumption
in the context of our model. Higher energy (more massive) mesons were
created momentarily in the Big Bang but are not thought to play a role
in nature today. Such particles are also regularly created in
experiments; for instance the heaviest meson created is the upsilon
meson with mass $9.46 \frac{GeV}{c^2}$ (roughly $10$ times the mass of
the proton/neutron). However these heavy mesons do not play a role in
the model described by Dirac-Klein-Gordon Equations.
  
We now turn our attention to the mathematical aspects of
\eqref{DKG}. The fundamental question is that of global regularity of
solutions. For smooth and small initial data endowed with additional
algebraic structure, Chadam and Glassey \cite{ChGl} established global
regularity for solutions of \eqref{DKG}. The work of Klainerman
\cite{Kl} on nonlinear Klein-Gordon equations paved the way of
establishing a more general result. Following those ideas and taking
advantage of the null structure present in the system, Bachelot
\cite{Ba} established global regularity for (very) smooth and small
initial data.  The next direction of research was to obtain a local in
time result for rough data as close as possible to the critical space
which is
\[\psi_0 \in L^2 ,\quad (\phi_0,\phi_1) \in H^{\frac12} \times
H^{-\frac12}.\]
 
Beals and Bezard \cite{BeBe} proved that for small initial data
$(\phi_0,\phi_1) \in H^2 \times H^1, \psi_0 \in H^1$ one has a local
well-posedness theory for \eqref{DKG}. Bournaveas in \cite{Bo}
improved this local in time result to $(\phi_0,\phi_1) \in
H^{1+\epsilon} \times H^{\epsilon}, \psi_0 \in H^{\frac12+\epsilon}$,
for any $\epsilon > 0$. In \cite{DaFoSe} D'Ancona,
Foschi and Selberg  established local
well-posedness of \eqref{DKG} for data $(\phi_0,\phi_1) \in
H^{\frac12+\epsilon} \times H^{-\frac12+\epsilon}, \psi_0 \in
H^{\epsilon}$, for any $\epsilon > 0$; hence the last result covers
the full subcritical regime.

Recently, Wang \cite{Wa} proved a global in time result for small initial data
in the critical Besov space $(\phi_0,\phi_1) \in \dot
B^{\frac12}_{2,1} \times \dot B^{-\frac12}_{2,1}, \psi_0 \in
\dot B^0_{2,1}$ (for $M=m=0$), additionally assuming that an angular derivative is bounded in the same space;
the proof exploits the observation of Sterbenz \cite{St} that angular
regularity acts as a null-structure. The result is then extended to
non-zero masses under the condition $2M>m>0$.

It is worth mentioning that in all of the above results the masses
$M,m$ are arbitrary; the result in \cite{Wa} is an exception. In the
context of a local in time result, the terms $M \psi$, $m^2 \phi$ can
be treated as perturbations, thus allowing an analysis of \eqref{DKG}
as a system of wave equations. Obviously, this cannot be the case for a
global in time theory which includes scattering.

In the context of the cubic Dirac system \cite{BH} we proposed a
different approach that incorporates the terms $M \psi$ and $m^2 \phi$
into the linear part of the operator, as they naturally appear. This
will help us treat \eqref{DKG} as a system of (half) Klein-Gordon
equations after using projectors which are adapted to our context from
the work of D'Ancona, Foschi and Selberg \cite{DaFoSe}. Then we
restrict our attention to the physical relevant case $2M > m >0$ and obtain
a global (in time) result and scattering for small initial data in the
subcritical regime. The resolution spaces used here have a simpler structure compared to \cite{BH}.
Our main result is the following
\begin{thm}\label{thm:main}
  Assume that $\epsilon > 0$ and $2M > m >0$. Then the Cauchy problem
  \eqref{DKG}-\eqref{eq:i-cond} is globally well-posed for small
  initial data
  \[\psi_0 \in H^{\epsilon}(\R^3;\C^4),\; (\phi_0,\phi_1) \in
  H^{\frac12+\epsilon}(\R^3;\R) \times H^{-\frac12+\epsilon}(\R^3;\R)\] and
  these solutions scatter to free solutions for $t\to \pm \infty$.
\end{thm}
We refer to Subsection \ref{subsect:proof} for more details.
Our result is at the same level of regularity as the one proved by
D'Ancona, Foschi and Selberg \cite{DaFoSe}.  Its strength lies in the
global in time and scattering parts. In terms of Sobolev regularity
it is slightly more restrictive than Wang's result \cite{Wa}. However,
we do not assume additional angular regularity on the initial data,
cp.\ also Remark \ref{rmk:ang-reg}.

A key observation is that under the assumption $2M > m >0$ the system
\eqref{DKG} has no resonances. It was known from prior works on
Klein-Gordon type systems with multiple speeds that, under certain
conditions between the masses, resonant interactions do not occur and
the well-posedness theory improves. We refer the reader to the works
of Delort and Fang \cite{DeFa}, Schottdorf \cite{S12} and Germain
\cite{Ge} and to the references therein. We will use this, together
with some localized Strichartz estimates, to prove the key nonlinear
estimates.

Note that unlike many of the previous works which dealt with power
type nonlinearities for the Klein-Gordon equation, the
Dirac-Klein-Gordon system contains derivatives. This is not apparent
from our formulation of \eqref{DKG}; however if one wants to write
\eqref{DKG} as a system of Klein-Gordon equations, one should apply
$(-i \gamma^\mu \partial_\mu - M )$ to the first equation and then it
is obvious that the right hand side contains derivatives.

We conclude this section with an overview of the paper. In Section \ref{sect:red}
we introduce some of the basic notation and rewrite the original system \eqref{DKG} in 
the equivalent form \eqref{DKGf} which has two advantages: it is first order in time and it
unveils the null structure. The gains from the null structure are quantified in Subsection \ref{subsect:mod}
in a manner that fits our analysis. In Section \ref{sect:fs-linear} we define the resolution space in which we iterate
our system. Without getting into technical details at this point, there is one particular aspect of this section that deserves to be highlighted. 
Proving Strichartz estimates has become a standard type argument due to the Christ-Kiselev Lemma \cite{CK}. However, proving localized versions of the Strichartz estimates using Christ-Kiselev type arguments is not straightforward. In Section \ref{sect:fs-linear} we provide an alternative argument for establishing (localized) Strichartz estimates using $U^p,V^p$ spaces and we think that this part of the paper may be of independent interest. In Section \ref{sect:nl} we prove the trilinear estimates based on which we prove our main result in Theorem \ref{thm:main}.

\section{Reductions}\label{sect:red}

\subsection{Notation}\label{subsect:not}
We define $A\ls B$, if there is a harmless constant $c>0$ such that $A\leq c B$, and $A\gs B$ iff $B\ls A$. Further, we define $A\approx B$ iff both $A\ls B$ and $B\ls A$. Also, we define $A\ll B$ if the constant $c$ can be chosen such that $c<2^{-10}$. Also, $A\gg B$ iff $B\ll A$.

Similarly, we define $A\preceq B$ iff $2^A\ls 2^B$, $A\succeq B$ iff $2^A\gs 2^B$, $A \sim B$ iff $2^A\approx 2^B$, $A\prec B$ iff $2^A\ll 2^B$, $A\succ B$ iff $2^A\gg 2^B$.

Let $\rho^0\in C^\infty_c(-2,2)$ be a fixed smooth, even, cutoff
satisfying $\rho^0(s)=1$ for $|s|\leq 1$ and $0\leq \rho\leq 1$. For
$k \in \Z$ we define $\rho_k:\R^3 \to \R$,
$\rho_k(y):=\rho^0(2^{-k}|y|)-\rho^0(2^{-k+1}|y|)$, such that
$A_k:=\supp(\rho_k)\subset \{y \in \R^3 \colon 2^{k-1}\leq |y|\leq
2^{k+1}\}$. Let $\tilde{\rho}_k=\rho_{k-1}+\rho_k+\rho_{k+1}$ and
$\tilde{A}_k:=\supp(\tilde{\rho}_k)$. For $k \geq 1$, let $P_k$ be the
Fourier multiplication operators with respect to $\rho_k$, and
$P_0=I-\sum_{k \geq 1}P_k$.  For $j \in \Z$ we define
\[
\mathcal{F}[Q^{\pm,m}_{j}f](\tau,\xi)=\rho_j(\tau\pm \la\xi
\ra_m)\mathcal{F}f(\tau,\xi).
\]
Similarly, we define $\tilde{P}_k$ and $\tilde{Q}^{\pm,m}_{j}$.

We also define $P_{\leq k}=\sum_{0\leq k'\leq k }P_{k'}$, $P_{\prec
  k}=\sum_{0\leq k'\prec k }P_{k'}$, $P_{> k}=I-P_{\leq k}$,
$P_{\succeq k}=I-P_{\prec k}$, and similarly $Q^{\pm,m}_{\leq j}$,
$Q^{\pm,m}_{\prec j}$, $Q^{\pm,m}_{\succeq j}$, and $Q^{\pm,m}_{j \in J}$ for an interval $J$.
In the obvious way we also define the analogous operators based on $\tilde{P}_k$ and $\tilde{Q}^{\pm,m}_{j}$.

In the case $m=1$ we suppress the superscripts, e.g.\
$Q^{\pm,1}_{j}=Q^{\pm}_{j}$.

Further, for $l \in \N$ let $\mathcal{K}_{l}$ denote a set of
spherical caps of radius $2^{-l}$ which is a covering of $\S^2$ with
finite overlap. For a cap $\kappa \in \mathcal{K}_{l}$ we denote its
center in $\S^2$ by $\omega(\kappa)$. Let $\Gamma_\kappa$ be the cone
generated by $\kappa \in \mathcal{K}_{l}$ and $(\eta_\kappa)_{\kappa
  \in \mathcal{K}_l}$ be a smooth partition of unity subordinate to
$(\Gamma_\kappa)_{\mathcal{K}_{l}}$. Let $P_{\kappa}$ denote the
Fourier-muliplication operator with symbol $\eta_\kappa$, such that
$I=\sum_{\kappa \in \mathcal{K}_l}P_{\kappa}$. Further, let
$\tilde{P}_{\kappa}$ with doubled support such that
$P_{\kappa}=\tilde{P}_{\kappa}P_{\kappa}=P_{\kappa}\tilde{P}_{\kappa}$.
For notational convenience, we also define $\mathcal{K}_{0}=\{\S^2\}$
and $P_\kappa=I$ if $\kappa \in \mathcal{K}_{0}$.
\subsection{Setup of the system and null
  structure}\label{subsect:setup-null}

As written in \eqref{DKG} the cubic Dirac-Klein-Gordon system has a
linear part whose coefficients are matrices and it is technically
easier to work with scalar equations.  To do so, we adapt the setup
introduced in \cite[Section 2 and 3]{DaFoSe} to take into account the
mass terms, similarly to our prior work on the cubic Dirac equation
\cite{BH} (however, the sign convention is in accordance with
\cite{DaFoSe}). We repeat here the essential steps for convenience of
the reader. As highlighted in \cite{DaFoSe} the new setup is able to
identify a null-structure in the nonlinearity, although the presence
of mass terms alters the effectiveness of this structure at very small
scales.

For $j=1,2,3$ the matrices
$
\alpha^j:=\gamma^0\gamma^j, \; \beta:=\gamma^0
$
have the properties
\[
\alpha^j\beta+\beta\alpha^j=0, \;\alpha^j \alpha^k+\alpha^k\alpha^j=2\delta^{jk}I_4,
\]
see \cite[p.~878]{DaFoSe} for more details.

We introduce the Fourier multiplication operators $\Pi_{\pm}^M(D)$
with symbol
\[
\Pi^M_\pm(\xi)=\frac12 [I \pm \frac{1}{\la \xi \ra_M} ( \xi \cdot
\alpha + M \beta)]
\]
In the case $M=1$ we suppress the superscript, i.e.\
$\Pi_{\pm}(D)=\Pi_{\pm}^1(D)$.

We then define $\psi_\pm=\Pi_\pm^M(D) \psi$ and split $\psi=\psi_+ +
\psi_-$. Also, define $\la D \ra=\sqrt{1-\Delta}$. By applying the operators $\Pi_\pm^M(D)$ to the system
\eqref{DKG} we obtain the following system of equations:
\begin{equation} \label{DKGnew} \left\{
    \begin{aligned} 
      & (-i\partial_t + \langle D \rangle_M) \psi_+ = \Pi^M_+(D) (\phi \beta \psi) \\
      & (-i\partial_t - \langle D \rangle_M) \psi_- = \Pi^M_-(D) ( \phi \beta \psi)  \\
      & (\Box + m^2) \phi = \la \psi, \beta \psi \ra.
    \end{aligned}
  \right.
\end{equation}
In order to have a fully first order system, we define
$\phi_\pm=\phi\pm i \la D \ra_m^{-1} \partial_t \phi$ thus
\[
(-i \partial_t + \la D \ra_m) \phi_+ = \la D \ra_m^{-1} \la \psi,
\beta \psi \ra.
\]
Note that $\phi=\Re \phi_\pm$ and $\phi_-=\overline{\phi_+}$ since
$\phi$ is real-valued.  The system which we will study is
\begin{equation} \label{DKGf} \left\{
    \begin{aligned} 
      & (-i\partial_t + \langle D \rangle_M) \psi_+ = \Pi^M_+(D) (\Re \phi_+ \beta \psi) \\
      & (-i\partial_t - \langle D \rangle_M) \psi_- = \Pi^M_-(D) ( \Re \phi_+ \beta \psi)  \\
      & (-i \partial_t +\la D \ra_m) \phi_+ = \la D \ra_m^{-1} \la
      \psi, \beta \psi \ra.
    \end{aligned}
  \right.
\end{equation}
We aim to provide a global theory for this system for initial data
$(\psi_{\pm,0}, \phi_{+,0}) \in H^\epsilon \times H^{\frac12 +
  \epsilon}$.  It is an easy exercise that this translates back into a
global theory for the original system with $(\psi_0,\phi_0,\phi_1) \in
H^\epsilon \times H^{\frac12 + \epsilon} \times H^{-\frac12 +
  \epsilon}$.

There is a null structure in the system \eqref{DKGf}, which we
describe next.  This is again inspired by the work in \cite{DaFoSe}
and was adapted to the current setup in \cite{BH}. For more details,
we refer to the reader to \cite{DaFoSe,BH}.

We decompose $\la \psi, \beta \psi \ra$ as
\[
\begin{split}
  \la \psi, \beta \psi\ra & = \la \Pi^M_+(D) \psi_+, \beta \Pi^M_+(D) \psi_+ \ra + \la \Pi^M_-(D) \psi_-, \beta \Pi^M_-(D)) \psi_- \ra \\
  &\quad + \la \Pi^M_+(D) \psi_+, \beta \Pi^M_-(D) \psi_- \ra+ \la
  \Pi^M_-(D) \psi_-, \beta \Pi^M_+(D) \psi_+ \ra.
\end{split}
\]
We have
\begin{equation}\label{eq:com}
  \Pi^M_\pm(D)\beta=\beta\Pi^M_\mp(D)\pm M \la D\ra^{-1}_M \beta 
\end{equation}
The following Lemma, which corresponds to \cite[Lemma 3.1]{BH} and
\cite[Lemma 2]{DaFoSe}, analyses the symbols of the bilinear operators
above.
\begin{lem}\label{lem:PiPi} For fixed $M \geq 0$, the following holds true:
  \begin{equation} \label{PiPi}
    \begin{split}
      \Pi_\pm^M(\xi) \Pi_\mp^M(\eta) & = \mathcal{O}(\angle (\xi,\eta)) + \mathcal{O}(\la \xi \ra^{-1} + \la \eta \ra^{-1}) \\
      \Pi_\pm^M(\xi) \Pi_\pm^M(\eta) & = \mathcal{O}(\angle
      (-\xi,\eta)) + \mathcal{O}(\la \xi \ra^{-1} + \la \eta \ra^{-1})
    \end{split}
  \end{equation}
\end{lem}

We now explain heuristically why this is useful here, see Lemma
\ref{lem:stable} for the technical result which will be used in the
nonlinear analysis. By \eqref{eq:com} it follows that for $s_1,s_2 \in \{+,-\}$
\begin{align*}
  \mathcal{F}_x\la \Pi_{s_1} \psi_1,\beta\Pi_{s_2} \psi_2\ra(\xi)&=\int\limits_{\xi=\xi_1-\xi_2}\la \Pi_{s_1}(\xi_1) \widehat{\psi_1}(\xi_1),\beta\Pi_{s_2}(\xi_2) \widehat{\psi_2}(\xi_2)\ra d\xi_1d\xi_2\\
  =&\int\limits_{\xi=\xi_1-\xi_2}\la \beta\Pi_{-s_2}(\xi_2)\Pi_{s_1}(\xi_1) \widehat{\psi_1}(\xi_1),\widehat{\psi_2}(\xi_2)\ra d\xi_1d\xi_2\\
  &+s_2M\int\limits_{\xi=\xi_1-\xi_2}
  \la\xi_1\ra_M^{-1}\la\beta\Pi_{s_1}(\xi_1) \widehat{\psi_1}(\xi_1),\widehat{\psi_2}(\xi_2)\ra
  d\xi_1d\xi_2.
\end{align*}
Hence, smallness of the angle $\angle(s_1\xi_1,s_2\xi_2)$ can be
exploited as long as it exceeds
$\max(\la\xi_1\ra_M^{-1},\la\xi_2\ra_M^{-1})$.  See
\cite[p.~885]{DaFoSe} for the analogue of this in the massless case,
where we have $\Pi_-^0(\xi_1) \Pi_+^0(\xi_2) =0$ if $\angle
(\xi_1,\xi_2)=0$, which makes the null structure effective at all
angular scales. In the massive case $M>0$ the null-structure does not
bring gains beyond $\max(\la\xi_1\ra_M^{-1},\la\xi_2\ra_M^{-1})$.  To
compensate for this we need to use that there are no resonances
present in \eqref{DKGf}.

In fact, as observed in \cite{DaFoSe}, there is a second and similar
null-structure in the nonlinearities present in the equations for
$\psi_\pm$ which will be exploited by duality in Section
\ref{sect:nl}.

\subsection{Modulation analysis}\label{subsect:mod}
A key aspect in the nonlinear analysis is the lack of resonant
terms. Arguments of similar nature are contained in \cite[Lemma
2]{S12}, see also \cite{DeFa,Ge}. Additionally, we will prove that
smallness of the maximal modulation induces angular constraints. In
the context of the cubic Dirac equation a similar result is contained
in \cite[Lemma 6.5]{BH}. We first provide lower bounds for the
resonance function.
\begin{lem}\label{lem:res}
  Fix $0<m<2M$. For $s_1,s_2\in \{+,-\}$ define the resonance function
  \begin{equation}\label{eq:res-fct}
    \mu^{s_1,s_2}(\xi_1,\xi_2):=\la\xi_1-\xi_2\ra_m+s_1\la \xi_1\ra_M-s_2\la \xi_2\ra_M.
  \end{equation}
  Then, we have the following bounds:

  {\it Case 1:} If
  \begin{enumerate} \item[a)]$s_1=+, s_2=-$ or
  \item[b)] $s_1=-,s_2=+$ and $\la\xi_1-\xi_2\ra_m\ll \min(\la
    \xi_1\ra_M, \la \xi_2\ra_M)$,
  \end{enumerate}
  then
  \begin{equation}\label{eq:high-mod}
    |\mu^{s_1,s_2}(\xi_1,\xi_2)|\gs \max(\la \xi_1-\xi_2\ra,\la \xi_1\ra,\la \xi_2\ra)
  \end{equation}

  {\it Case 2:} If \begin{enumerate} \item[a)] $s_1=s_2$ or
  \item[b)] $s_1=-,s_2=+$ and $\la\xi_1-\xi_2\ra_m\gs \min(\la
    \xi_1\ra_M, \la \xi_2\ra_M)$,
  \end{enumerate}
  then
  \begin{equation}\label{eq:mod-angle}
    \begin{split}
      |\mu^{s_1,s_2}(\xi_1,\xi_2)|\gs_{m,M} &\frac{\la \xi_1\ra \cdot
        \la \xi_2\ra}{\la \xi_1-\xi_2\ra}\angle(s_1\xi_1,s_2\xi_2)^2
    \end{split}
  \end{equation}
  With any choice of signs, we have both
  \begin{equation}\label{eq:gen-lb}
    |\mu^{s_1,s_2}(\xi_1,\xi_2)|\gs_{m,M}\min(\la \xi_1\ra , \la \xi_2\ra )\angle(s_1\xi_1,s_2\xi_2)^2,
  \end{equation}
  and the \emph{non-resonance} bound
  \begin{equation}\label{eq:non-res}
    |\mu^{s_1,s_2}(\xi_1,\xi_2)|\gs_{m,M} \max(\la \xi_1-\xi_2\ra^{-1},\la \xi_1\ra^{-1},\la \xi_2\ra^{-1}).
  \end{equation}
\end{lem}
\begin{proof}
  In Case 1 the lower bound \eqref{eq:high-mod} is obvious, which
  implies all other claims.

  Suppose now that we are in Case 2 a):
  \begin{align*}
    &(\la\xi_1-\xi_2\ra_m-|\la \xi_1\ra_M-\la \xi_2\ra_M|)(\la\xi_1-\xi_2\ra_m+|\la \xi_1\ra_M-\la \xi_2\ra_M|)\\
    &=2(|\xi_1||\xi_2|-\xi_1\cdot\xi_2)+m^2+2(\la\xi_1\ra_M\la\xi_2\ra_M-|\xi_1||\xi_2|-M^2)
  \end{align*}
  Now, we compute
  \begin{equation}\label{eq:d}\la\xi_1\ra_M\la\xi_2\ra_M-(|\xi_1||\xi_2|+M^2)=M^2\frac{(|\xi_1|-|\xi_2|)^2}{\la\xi_1\ra_M\la\xi_2\ra_M+|\xi_1||\xi_2|+M^2}
  \end{equation}
  Since this is non-negative, we conclude
  \begin{align*}
    &(\la\xi_1-\xi_2\ra_m-|\la \xi_1\ra_M-\la \xi_2\ra_M|)(\la\xi_1-\xi_2\ra_m+|\la \xi_1\ra_M-\la \xi_2\ra_M|)\\
    &\geq 2|\xi_1||\xi_2|(1-\cos\angle(\xi_1,\xi_2))+m^2\\
    &\gs |\xi_1||\xi_2|\angle(\xi_1,\xi_2)^2+m^2
  \end{align*}
  Now, because of $m>0$ and $\la\xi_1-\xi_2\ra_m+|\la \xi_1\ra_M-\la
  \xi_2\ra_M|\ls \la\xi_1-\xi_2\ra_m$ the estimates \eqref{eq:gen-lb}
  and \eqref{eq:mod-angle} follow. Also, \eqref{eq:non-res} follows if
  $\la\xi_1-\xi_2\ra\ls \min(\la \xi_1\ra,\la \xi_2\ra)$. Otherwise,
  we have $\max(\la \xi_1\ra,\la \xi_2\ra)\gg\min(\la \xi_1\ra,\la
  \xi_2\ra)$, and the estimate \eqref{eq:non-res} follows from
  \begin{align*}
    &(\la\xi_1-\xi_2\ra_m-|\la \xi_1\ra_M-\la \xi_2\ra_M|)(\la\xi_1-\xi_2\ra_m+|\la \xi_1\ra_M-\la \xi_2\ra_M|)\\
    &\geq
    M^2\frac{(|\xi_1|-|\xi_2|)^2}{\la\xi_1\ra_M\la\xi_2\ra_M+|\xi_1||\xi_2|+M^2},
  \end{align*}
  where we used \eqref{eq:d} again.

  Suppose now that we are in Case 2 b): A computation similar to the
  above yields
  \begin{align*}
    &(\la \xi_1\ra_M+\la \xi_2\ra_M-\la\xi_1-\xi_2\ra_m)(\la \xi_1\ra_M+\la \xi_2\ra_M+\la\xi_1-\xi_2\ra_m)\\
    &=2(|\xi_1||\xi_2|+\xi_1\cdot\xi_2)+2M^2-m^2+2(\la\xi_1\ra_M\la\xi_2\ra_M-|\xi_1||\xi_2|)\\
    &\gs |\xi_1||\xi_2|\angle(-\xi_1,\xi_2)^2+4M^2-m^2.
  \end{align*}
  By assumption $4M^2-m^2>0$, so the estimate \eqref{eq:gen-lb} is
  proved, and due to $\la\xi_1-\xi_2\ra\approx \max(\la \xi_1\ra, \la
  \xi_2\ra)$ the claim \eqref{eq:mod-angle} follows, too. Also, if $|
  \xi_1|\approx | \xi_2|$, \eqref{eq:non-res} follows. Otherwise, we use
  the lower bound provided by \eqref{eq:d} to obtain
  \eqref{eq:non-res}.
\end{proof}

\begin{rmk}\label{rmk:m}
From now on we fix $M=m=1$ in oder to simplify the exposition. In view
of Lemma \ref{lem:res} it will be obvious that all arguments carry
over to the case $2M>m>0$ with modified (implicit) constants depending
on $m,M$.
\end{rmk}

\begin{lem} \label{lem:mod} Let $s_1,s_2\in \{+,-\}$. Consider
  $k,k_1,k_2\in \N_0, j,j_1,j_2 \in \Z$, and $\phi=\tilde{P}_{k}
  \tilde{Q}^{+}_{j}\phi$, $u_i=\tilde{P}_{k_i}
  \tilde{Q}^{s_i}_{j_i}u_i$.

  i) If $\max(j,j_1,j_2) \prec -\min(k,k_1,k_2)$, we have
  \begin{equation}\label{eq:int-zero}
    \int_{\R^{1+3}}\phi \cdot u_1\overline{u_2} \, dtdx =0.
  \end{equation}

  ii) {\it Case 1:} Suppose that \begin{align*}& s_1=+,s_2=-\\
\text{ or }\qquad & s_1=-,s_2=+ \text{ and }k\prec \min(k_1,k_2).\end{align*}
If $\max(j,j_1,j_2)\prec \max(k,k_1,k_2)$,
  then, \eqref{eq:int-zero} holds true.

  {\it Case 2:} Suppose that \begin{align*}&s_1=s_2\\
\text{ or }\qquad & s_1=-,s_2=+ \text{ and }k\succeq
  \min(k_1,k_2).
\end{align*}
If $l\geq 1$, $\kappa_1,\kappa_2 \in
  \mathcal{K}_{l}$ with $\dist(s_1\kappa_1,s_2\kappa_2)\geq 2^{-l}$
  and $\max(j,j_1,j_2)\prec k_1+k_2-k-2l$, then
  \begin{equation}\label{eq:int-zero-caps}
    \int_{\R^{1+3}}\phi \cdot \tilde{P}_{\kappa_1} u_1\overline{\tilde{P}_{\kappa_2}u_2} \, dtdx =0.
  \end{equation}
\end{lem}
\begin{proof} We have
  \[
  \int_{\R^{1+3}}\phi \cdot u_1\overline{u_2} \, dtdx =
  \int_{\R^{1+3}}\widehat{\phi} \overline{
    \widehat{\overline{u_1}u_2}} \, d\tau d\xi
  \]
  and, with $\zeta=(\tau,\xi)$,
  \[
  \widehat{\overline{u_1}u_2}(\zeta)=\int
  \widehat{\overline{u_1}}(\zeta')\widehat{u_2}(\zeta-\zeta')d\zeta'
  =\int
  \overline{\widehat{u_1}}(-\zeta')\widehat{u_2}(\zeta-\zeta')d\zeta',
  \]
  hence, with $\zeta_j=(\tau_j,\xi_j)$,
  \begin{equation}\label{eq:ft}
    \int_{\R^{1+3}}\phi \cdot u_1\overline{u_2} \, dtdx
    =
    \int\int\widehat{\phi}(\zeta_2-\zeta_1)\widehat{u_1}(\zeta_1)\overline{\widehat{u_2}}(\zeta_2)d\zeta_1d\zeta_2
  \end{equation}
  The assumptions imply that we must have
  \[
  |\tau_2-\tau_1+\la\xi_2-\xi_1\ra|\approx 2^{j}, \;
  |\tau_1+s_1\la\xi_1\ra|\approx 2^{j_1}, \; |\tau_2+s_2\la\xi_2\ra|\approx
  2^{j_2}
  \]
  in order to obtain a nontrivial contribution. This implies
  \begin{equation}\label{eq:up-res}
    |\la\xi_2-\xi_1\ra+s_1\la\xi_1\ra-s_2\la\xi_2\ra|\ls
    2^{\max(j,j_1,j_2)}.
  \end{equation}

  i) By assumption we have $2^{\max(j,j_1,j_2)}\ll
  2^{-\min(k,k_1,k_2)}$, so that \eqref{eq:up-res} contradicts
  \eqref{eq:non-res}.

  ii) By assumption we have $2^{\max(j,j_1,j_2)}\ll
  2^{\max(k,k_1,k_2)}$ in Case 1, hence \eqref{eq:up-res} contradicts
  \eqref{eq:high-mod}. Similarly, in Case 2 the estimate \eqref{eq:up-res}
  contradicts \eqref{eq:mod-angle}.
\end{proof}

\section{Function spaces and linear estimates}\label{sect:fs-linear}
For $1\leq p\leq \infty$, $b \in \R$, we define
\[\|f\|_{\dot{X}^{\pm,b,p}}=\big\|\big(2^{bj}\|Q_{j}^\pm
f\|_{L^2}\big)_{j \in \Z }\big\|_{\ell^p},\]

The low frequency part will be treated altogether, that is we define
\[
\| f \|_{S^\pm_{\leq 0}} = \| f \|_{L^\infty_t L^2_x} + \| f \|_{L^2_t
  L^6_x} + \|f \|_{\dot{X}^{\pm,\frac12,\infty}}.
\]
By interpolation, the space above provides all the Strichartz
estimates for the Schr\"odinger equation on $\R^3$. This is
natural since the Klein-Gordon equation in low frequency behaves like the
Schr\"odinger equation.

In high frequency, the Klein-Gordon equation is of wave type and the
Strichartz estimates should reflect that.  Moreover we need some
refinement of the standard Strichartz estimates.

For $d=3$ and  $k\in \Z_+$ let $\Xi_k=2^k\cdot\Z^d$. Let
$\gamma^{(1)}:\mathbb{R}\to[0,1]$ denote an even smooth function
supported in the interval $[-2/3,2/3]$ with the property that
\begin{equation*}
  \sum_{n\in\Z}\gamma^{(1)}(\xi-n)= 1\text{ for }\xi \in \R.
\end{equation*}
Let $\gamma:\mathbb{R}^d\to[0,1]$,
$\gamma(\xi)=\gamma^{(1)}(\xi_1)\cdot\ldots\cdot
\gamma^{(1)}(\xi_d)$. For $k\in \Z_+$ and $n\in\Xi_k$ let
\begin{equation*}
  \gamma_{k,n}(\xi)=\gamma((\xi-n)/2^k).
\end{equation*}
Clearly, $\sum_{n\in\Xi_k}\gamma_{k,n}\equiv 1$ on $\R^d$. Now, we
define the Fourier-multiplication operators $\Gamma_{k,n}$ with symbol
$\gamma_{k,n}$.

There is the following refinement of the classical Strichartz
estimate. In the context of Strichartz-Pecher inequalities for the
wave equation, the underlying decay estimate after localization to
cubes has been proved in \cite[(A.59)]{KT99}, see also \cite[Theorem
4.1]{S08} for the case $p=q=4$.
\begin{lem}\label{lem:loc-str}
  Let $d=3$, $\frac{1}{p}+\frac{1}{q}=\frac{1}{2}$ with $p>2$. Then,
  \begin{equation}\label{eq:loc-str}
    \sup_{0 \leq k' \leq k} 2^{-\frac{k'+k}{p}} \left( \sum_{n \in
        \Xi_{k'}} \| \Gamma_{k',n} P_k e^{\pm it \la D \ra } f
      \|^2_{L^p_t L^q_x} \right)^{\frac12}\ls \|f\|_{L^2(\R^3)}
  \end{equation}
\end{lem}
\begin{proof}
  By orthogonality, it suffices to prove
  \[
  \| \Gamma_{k',n} P_k e^{\pm it \la D \ra } f \|_{L^p_t L^q_x} \ls
  2^{\frac{k'+k}{p}}\|f\|_{L^2(\R^3)},
  \]
  uniformly in $n\in \Xi_{k'}$. Let $T=\Gamma_{k',n} P_k e^{\pm it \la
    D \ra }$. The operator $TT^\ast$ is a space-time convolution
  operator with the kernel
  \[
  K_{k',k;n}(t,x)=\int_{\R^3} e^{\pm it\la \xi \ra +ix\cdot
    \xi}\rho_k^2(\xi)\gamma_{k',n}^2(\xi)d\xi.
  \]
  By the $TT^\ast$-argument, it suffices to prove
  \[
  \|TT^\ast\|_{L^{p'}_t L^{q'}_x \to L^p_t L^q_x}\ls
  2^{\frac{2(k+k')}{p}}
  \]
  which reduces to proving the kernel bound
  \begin{equation}
    \label{eq:ker}
    |K_{k',k;n}(t,x)|\ls 2^{3k'} \big(1+2^{2k'-k}|t|\big)^{-1}.
  \end{equation}
  Indeed, by interpolation and Young's inequality, we obtain
  \[
  \|K_{k',k;n}(t,\cdot)\ast \phi \|_{L^q_x(\R^3)}\ls 2^{3k'
    (1-\frac{2}{q})}
  \big(1+2^{2k'-k}|t|\big)^{-(1-\frac{2}{q})}\|\phi\|_{L^{q'}_x(\R^3)},
  \]
  and Hardy-Littlewood-Sobolev with
  $\frac{1}{r}=\frac{2}{p}=1-\frac{2}{q}$ implies
  \begin{align*}
    \|TT^\ast\|_{L^{p'}_t L^{q'}_x \to L^p_t L^q_x} &\ls 2^{3k'
      (1-\frac{2}{q})}
    \|\big(1+2^{2k'-k}|t|\big)^{-(1-\frac{2}{q})}\|_{L^{r,w}_t}\\
    &\ls 2^{ \frac{2}{p}(k+k')}.
  \end{align*}
  Finally, we give a proof of \eqref{eq:ker}: Rescaling yields
  \[
  K_{k',k;n}(t,x)=2^{3k}K_{k'-k,1,2^{-k}n}(2^k t, 2^k x),
  \]
  where, for $\la \xi \ra_k:=(|\xi|^2+2^{-2k})^{\frac12}$,
  \[
  K_{j,1,a}(s,y)=\int_{\R^3} e^{\pm is\la \xi \ra_k +iy\cdot
    \xi}\rho_1^2(\xi)\gamma_{j,a}^2(\xi)d\xi
  \]
  For $|a|\approx 1$, we claim
  \begin{equation}\label{eq:red-ker}
    |K_{j,1,a}(s,y)|\ls 2^{3j}(1+2^{2j}|s|)^{-1}.
  \end{equation}
  For $|s|\leq 2^{-2j}$ this is immediate because the domain of
  integration has volume $2^{3j}$, and in the remaining case it can be
  proved as for the wave equation in \cite[(A.70)]{KT99}. We provide
  an explicit proof: By a simple covering argument we may replace
  $\rho_1^2\gamma_{j,a}^2$ by a smooth cutoff $\zeta$ with respect to
  a thickened spherical cap of size $2^j$ and denote the corresponding
  kernel by $\tilde{K}_{j,a}$. By rotation, we may assume that
  $y=(0,0,|y|)$. We use spherical coordinates:
  \[
  \tilde{K}_{j,a}(s,y)=\int_{0}^\infty \int_0^{2\pi} \int_0^\pi
  e^{i(|y|\rho \cos\theta+s\la \rho \ra_k) } \zeta(\theta,\varphi,\rho)
  \sin(\theta) \rho^2 d\theta d\varphi d\rho.
  \]
  We may choose
  $\zeta(\varphi,\theta,\rho)=\zeta_1(\theta)\zeta_2(\varphi)\zeta_3(\rho)$. The
  phase of the oscillatory integral is stationary only if $|y|\approx
  |s|$ and the cap is centered near the north pole or south pole,
  otherwise we get arbitrarily fast decay. We discuss only the first
  case, where we may further assume that $|\zeta_1' |\ls 2^{-j}$,
  $\zeta_1$ is supported in an interval of length $\ls 2^j$ in
  $[0,\pi)$, and $\zeta_3$ is supported in an interval of length $\ls
  2^j$ in $(1/4,4)$, with $|\zeta_3' |\ls 2^{-j}$.  We integrate by
  parts with respect to $\theta$:
  \begin{align*}
    \tilde{K}_{j,a}(s,y)& =\frac{i \zeta_1(0)}{|y|}\int_{0}^\infty
    \int_0^{2\pi} e^{i(|y|\rho+s\la \rho \ra_k) }
    \zeta_2(\varphi)\zeta_3(\rho)\rho d\varphi
    d\rho\\
    &- \frac{i}{|y|}\int_{0}^\infty \int_0^{2\pi} \int_0^\pi
    e^{i(|y|\rho\cos\theta+s\la \rho \ra_k) } \zeta_1'(\theta)d\theta
    \zeta_2(\varphi)\zeta_3(\rho)\rho d\varphi d\rho,
  \end{align*}
  and the properties of $\zeta_1$ and $\zeta_3$ imply
  \[
  |\tilde{K}_{j,a}(s,y)|\ls 2^j|y|^{-1},
  \]
  which completes the proof of \eqref{eq:red-ker}, which implies
  \eqref{eq:ker}.
\end{proof}

\begin{rmk}\label{rmk:gen-str}
  The generalization of Lemma \ref{lem:loc-str} to general dimension
  and non-sharp admissible pairs is obvious, but we do not need it
  here.
\end{rmk}

Now, we consider functions in $f\in L^\infty_t(\R;L^2(\R^3;\C^d)$. We will use $d=1$ for the Klein-Gordon part and $d=4$ for the Dirac part.
For $k \in \Z, k \geq 0$ and $l, k' \in \Z, 0 \leq k',l \leq k$, we define
\[
\| f \|_{L^p_t L^q_x[k;l,k']} : = \left(  \sum_{\kappa \in \mathcal{K}_l}\sum_{n \in \Xi_{k'}} \| \Gamma_{k',n}P_{\kappa}  f \|^2_{L^p_t L^q_x} \right)^\frac12
\]
Note that the above norm for $l=0$ is similar to the one in \eqref{lem:loc-str}. The general case $0 \leq l \leq k$ is needed for technical reasons. 

For $k \geq 0$, we define
\begin{equation}\label{eq:sk}
\begin{split}
  \| f \|_{S_k^\pm} & = \| f \|_{L^\infty_t L^2_x} + \| f \|_{\dot X^{\pm,\frac12,\infty}} \\
  & + \sup_{0 \leq k',l \leq k} \left( 2^{-\frac{k'+k}3} \| f \|_{L^3_t L^6_x[k;l,k']} + 2^{-\frac{k'+k}6} \| f \|_{L^6_t L^3_x[k;l,k']} \right).
\end{split}
\end{equation}
Note that if $k'=k$ and $l=0$, that is no additional localization is
provided, the last two norms are simply the standard Strichartz
estimates $L^3_t L^6_x$ and $L^6_t L^3_x$ available for the wave
equation in $\R^3$.

In the nonlinear estimates we will use that $\| P_{\leq 0} f
\|_{S_{\leq 0}^\pm}$ also dominates (by interpolation and the Sobolev
embedding) the localized Strichartz norms (with $k=0$) available in
the high frequency structure.

Next, we consider boundedness properties of certain multipliers.
\begin{lem}
  \label{lem:stable} {\it i)} Let $s_1,s_2\in \{+,-\}$.  For any
  $k_1,k_2\in \N_0$, $1\leq l\leq \min(k_1,k_2)+10$,
  $\kappa_1,\kappa_2\in \mathcal{K}_l$ with
  $\dist(s_1\kappa_1,s_2\kappa_2)\ls 2^{-l}$, $v_1,v_2\in \C^4$, we
  have
  \begin{equation}\label{eq:stable1}
    |\la \Pi_{s_1}(2^{k_1}\omega(\kappa_1))v_1,\beta\Pi_{s_2}(2^{k_2}\omega(\kappa_2))v_2\ra |\ls 2^{-l}|v_1||v_2|
  \end{equation}

 Fix $k \in \N_0$. All the statements below are made for functions localized at frequency $2^k$, i.e. they satisfy $f =\tilde{P}_{k} f$. 

  {\it ii)} For any $1\leq l\leq k+10$, $\kappa \in
  \mathcal{K}_l$, $f \in S^{\pm}_{k}$, we have
  \begin{equation}\label{eq:stable2}
    \|[\Pi_{\pm}(D)-\Pi_{\pm}(2^{k}\omega(\kappa))]P_{\kappa} f \|_{S^{\pm}_{k}}\ls 2^{-l}\|P_{\kappa} f \|_{S^{\pm}_{k}},
  \end{equation}
and similarly in $L^p_tL^q_x$-norms.

  {\it iii)} For any $j \in \Z$, the operators
  $Q_{j}^\pm $ are uniformly bounded on
  $S^\pm_k$.

{\it iv)} For any $l \in \N_0$, $\kappa\in
  \mathcal{K}_l$ and $j \in \Z$ with $j \geq k-2l-100$ the operators
  $Q_{>j}^\pm \tilde{P}_\kappa$ and $Q_{\leq
    j}^\pm \tilde{P}_\kappa$ are uniformly bounded on
  $S^\pm_k$.

  {\it v)} For any $k' \in \N_0$ and $j \in \Z$ satisfiying $k'\leq
  k $ and $j\geq 2k'-k$, the operators $Q_{>j}^\pm$ and $Q_{\leq
    j}^\pm$ are uniformly disposable in the sense that
  \begin{align*}
\sup_{0\leq l \leq k} \left(  2^{-\frac{k'+k}3}  \| Q_{{>j \atop [\leq j]}}^\pm f \|_{L^3_t L^6_x[k;l,k']} 
+ 2^{-\frac{k'+k}6} \| Q_{{>j \atop [\leq j]}}^\pm f \|_{L^6_t L^3_x[k;l,k']} \right)
    \ls \|f\|_{S_k^\pm}.
  \end{align*}
Further, similar estimates for  $Q_{>j}^\pm$ and $Q_{\leq
    j}^\pm$ hold with a bound $\la k' \ra$  as long as $j\succeq -k'$.
\end{lem}

\begin{proof}
  The identity \eqref{eq:com} implies
  \begin{align*}
    &\la \Pi_{s_1}(2^{k_1}\omega(\kappa_1))v_1,\beta\Pi_{s_2}(2^{k_2}\omega(\kappa_2))v_2\ra\\
    =&\la
    \beta\Pi_{-s_2}(2^{k_2}\omega(\kappa_2))\Pi_{s_1}(2^{k_1}\omega(\kappa_1))v_1,v_2\ra+s_2\la
    2^{k_2}\ra^{-1} \la\beta\Pi_{s_1}(2^{k_1}\omega(\kappa_1))v_1,v_2 \ra,
  \end{align*}
  hence \eqref{eq:stable1} follows from estimates \eqref{PiPi} and
  Cauchy-Schwarz.

  In order to prove \eqref{eq:stable2}, it suffices to consider the
  case of the $+$ sign. We write the matrix-valued symbol $p$ of
  $2[\Pi_{+}(D)-\Pi_{+}(2^{k}\omega(\kappa))]P_{k}\tilde{P}_{\kappa}$ as
  \begin{align*}
    p(\xi)&=2[\Pi_{+}(\xi)-\Pi_{+}(2^{k}\omega(\kappa))]\tilde{\rho}_k(\xi)\tilde{\eta}_\kappa(\xi)\\
    &=\tilde{\rho}_k(\xi)\tilde{\eta}_\kappa(\xi)\Big[\frac{\xi}{\la \xi\ra}-\frac{2^{k}\omega(k)}{\la 2^k\ra} \Big]\cdot \alpha+\tilde{\rho}_k(\xi)\tilde{\eta}_\kappa(\xi)\Big[\frac{1}{\la \xi\ra }-\frac{1}{\la 2^k\ra }\Big]\beta\\
    &=:p_1(\xi)+p_2(\xi)
  \end{align*}
  We further decompose
  \begin{align*}
    p_1(\xi)=&\tilde{\rho}_k(\xi)\tilde{\eta}_\kappa(\xi)\Big[\frac{|\xi|}{\la \xi\ra}-\frac{2^{k}}{\la 2^k\ra} \Big]\omega(\kappa)\cdot \alpha+\tilde{\rho}_k(\xi)\tilde{\eta}_\kappa(\xi)\frac{|\xi|}{\la \xi\ra}\Big[\frac{\xi}{|\xi|}-\omega(\kappa) \Big]\cdot \alpha\\
    &=:p_{11}(\xi)+p_{12}(\xi).
  \end{align*}
We denote the Fourier-multiplication operators defined by the symbols above by $P_2(D), P_{11}(D),P_{12}(D)$. Obviously, the properties of $\tilde{\rho}_k$ imply that
  \[\|P_2(D)\|_{L^p_x\to L^p_x} \ls 2^{-k}, \quad
  \|P_{11}(D)\|_{L^p_x\to L^p_x}\ls 2^{-k}, \text{ for any }1< p <
  \infty,\] and the properties of $\tilde{\eta}_\kappa$ imply that
  \[\|P_{12}(D)\|_{L^p_x\to L^p_x}\ls 2^{-l}, \text{ for any }1<p<
  \infty.\] The claim follows from the definition of the space
  $S^+_k$.

  Part iii) needs to be proved for the Strichartz norms only. For the
  operator $Q_j^\pm$ this is an easy consequence of the
  well-known transference principle. Indeed,
  \[Q_j^\pm f(t)=\int e^{it\tau} e^{\mp i t \la D\ra }
  \mathcal{F}_t(e^{\pm i t \la D\ra
  }f)(\tau)\tilde{\rho}_j(\tau)d\tau,\] hence by Lemma
  \ref{lem:loc-str} we obtain
  \begin{align*}
    &2^{-\frac{k'+k}{p}} \| Q_j^\pm f \|_{L^p_t L^q_x[k;l,k']}\\
    \ls{} & \|\mathcal{F}_t(e^{\pm i t \la D\ra }  f)\tilde{\rho}_j\|_{L^1_\tau L^2_{\xi}} 
    \ls 2^{\frac{j}{2}}\| \tilde{Q}_j^\pm f \|_{L^2}\approx \| f\|_{\dot
      X^{\pm,\frac12,\infty}}.
  \end{align*}

  In order to prove Part iv), we apply Sobolev inequalities to
  obtain for any $\kappa' \in \mathcal{K}_{l'}$, $n \in \Xi_{k'}$
  \begin{align*}
    &2^{-\frac{k'+k}{p}} \| \Gamma_{k',n}P_{\kappa'}  Q_j^\pm P_{\kappa}f \|_{L^p_t L^q_x}\\
    \ls & 2^{-\frac{k'+k}{p}}
    2^{j(\frac{1}{2}-\frac{1}{p})}2^{(k'+\min(2k-2l,2k'))(\frac{1}{2}-\frac{1}{q})}\|
    \Gamma_{k',n}P_{\kappa'}Q_j^\pm P_{\kappa}  f \|_{L^2}.
  \end{align*}
  Summing up the squares w.r.t.\ $\kappa' ,n$ yields
  \begin{equation}\label{eq:u}
       2^{-\frac{k'+k}{p}} \|  Q_j^\pm P_{\kappa}f \|_{L^p_t L^q_x[k;l',k']} 
      \ls  2^{\frac{\min(k-2l,2k'-k)-j}{p}} 2^{\frac{j}{2}}\| Q_j^\pm f \|_{L^2},
  \end{equation}
  which we finally sum up with respect to $j \geq j_0\geq k-2l-100$ to
  obtain
  \begin{align*}
     2^{-\frac{k'+k}{p}}  \| Q_{>j_0}^\pm P_{\kappa}f \|_{L^p_t L^q_x[k;l',k']} 
    \ls \| f \|_{\dot X^{\pm,\frac12,\infty}}
  \end{align*}
  The remaining claim in Part iv) follows from $Q_{\leq
    j}^\pm=I-Q_{>j}^\pm$.

  Part v) follows similarly from \eqref{eq:u}. The last claim for $Q_{>j}^\pm$ follows by applying Part iii) and Part v) to
\[
Q_{>j}^\pm=Q_{>2k'-k}^\pm+\sum_{j< j' \leq 2k'-k}Q_{j'}^\pm,
\]
because the number of terms in the second sum is bounded by $\la k'\ra$. The claim for $Q_{\leq j}^\pm=I-Q_{>j}^\pm$ follows, too.
\end{proof}

The next Lemma shows why the $S^\pm_k$-semi-norms  are useful in the context of the evolution equation.
\begin{lem}\label{lem:lin}
For any $k \in \N_0$, $u_0=\tilde{P}_ku_0\in L^2(\R^3;\C^d)$ and $f=\tilde{P}_kf\in L^1_t(\R,L^2(\R^3;\C^d))$, let
\[
u(t)=e^{\mp it \la D\ra}u_0 +i \int_0^t e^{\mp i(t-s) \la D\ra}f(s)ds.
\]
Then, $u=\tilde{P}_k u$ is the unique solution of
\[
-i\partial_t u \pm \la D \ra u=f,
\]
and $u\in C(\R,L^2(\R^3;\C^d))$ and
\begin{equation}\label{eq:lin}
\|u\|_{S^\pm_k}\ls{} \|u_0\|_{L^2(\R^3)}+\sup_{g \in G}\Big|\int_{\R^{1+3}}\la f, g\ra_{\C^d} dxdt \Big|
\end{equation}
provided that the right hand side of \eqref{eq:lin} is finite, where $G$ is defined as
the set of all $g=\tilde{P}_k g\in L^\infty_t (\R;L^2(\R^3;\C^d))$ such that $\|g\|_{S^{\pm}_k}=1$.
\end{lem}
\begin{proof}
Without the localization in $L^3_t L^6_x, L^6_t L^3_x$
the linear theory above is standard using $X^{s,b}$ theory and the Christ-Kiselev Lemma \cite{CK}. It is likely that one can adapt
the Christ-Kiselev Lemma to cover the localized versions of $L^3_t L^6_x, L^6_t L^3_x$ and their dual structures as well, but we do not pursue this strategy here. Instead, we will give a rather short proof using the theory of $U^2$ and $V^2$ spaces, see e.g.\ \cite{KT,HHK,KVT} for details.
We recall that for $1<p<\infty$ the atomic space $U^p_{\pm\la D\ra}$ is defined via its atoms
\[
a(t)=\sum_{k=1}^K \mathds{1}_{[t_{k-1},t_{k})}(t)e^{\mp i t \la D\ra} \phi_k , \quad \sum_{k=1}^K\|\phi_k\|_{L^2}^p=1,
\]
where $\{t_k\}$ is a partition, $t_K=+\infty$.

As a companion space we use the space $V^p_{\pm\la D\ra}$ of right-continuous functions $v$ such that $t\mapsto e^{\pm i t \la D\ra}v(t)$ is of bounded $p-$variation. We have $ V^2_{\pm\la D\ra}\hookrightarrow U^p_{\pm\la D\ra}$ for $p>2$.

For $0\leq l,k'\leq k$ we define
\begin{equation}\label{eq:loc-u2}
\|u\|_{U^\pm_{k;l,k'}}:=\Big(\sum_{\kappa \in \mathcal{K}_l}\sum_{n \in \Xi_{k'}} \| \Gamma_{k',n}P_{\kappa} u \|^2_{U^2_{\pm\la D\ra}}\Big)^{\frac12}.
\end{equation}
Then, we have
\begin{equation}\label{eq:loc-v2}
\|u\|_{V^\pm_{k;l,k'}}:=\Big(\sum_{\kappa \in \mathcal{K}_l}\sum_{n \in \Xi_{k'}} \| \Gamma_{k',n}P_{\kappa} u \|^2_{V^2_{\pm\la D\ra}}\Big)^{\frac12}\ls \|u\|_{U^\pm_{k;l,k'}}
\end{equation}
It is easy to show that the $U^\pm_{k;l,k'}$-norms are decreasing if we localize to smaller scales, i.e.\
\[
\|u\|_{U^\pm_{k;l,k'}}\ls \|u\|_{U^\pm_{k;\tilde{l},\tilde{k}'}} \text{ if } \tilde{l}\leq l \text{ and } \tilde{k'}\geq k',
\]
and the $V^\pm_{k;l,k'}$-norms are increasing if we localize to smaller scales, i.e.\
\[
\|u\|_{V^\pm_{k;l,k'}}\ls \|u\|_{V^\pm_{k;\tilde{l},\tilde{k}'}} \text{ if } \tilde{l}\geq l \text{ and } \tilde{k'}\leq k'.
\]
Set $U^\pm_{k}=U^\pm_{k;k,0}$ and $V^\pm_{k}=V^\pm_{k;k,0}$.

Strichartz estimates for admissible pairs $(p,q)$ hold for $U^p_{\pm\la D\ra}$-functions (which is easily verified for atoms), hence all for $V^2_{\pm\la D\ra}$-functions. For any $0\leq k',l\leq k$ we have
\[
2^{-\frac{k'+k}{p}} \Big(\sum_{\kappa \in \mathcal{K}_l}\sum_{n \in \Xi_{k'}} \| \Gamma_{k',n}P_{\kappa}  u \|^2_{L^p_tL^q_x}\Big)^{\frac12}
\ls{} \|u\|_{V^\pm_{k;l,k'}}\ls \|u\|_{V^\pm_{k}}.
\]

We also have $V^\pm_{k}\hookrightarrow V^2_{\pm\la D\ra}$ and $V^2_{\pm\la D\ra}$-norm dominates both the $L^\infty_t L^2_x$-norm and the $\dot{X}^{\pm,\frac12,\infty}$-seminorm. Hence,
\[\|u\|_{S^\pm_k}\ls \|u\|_{V^\pm_{k}}\ls \|u\|_{U^\pm_{k}}.\]
Now, we can use the $U^2$ duality theory (see e.g. \cite[Prop.~2.10]{HHK}, and \cite[Prop.~2.11]{HTT} for a frequency-localized version), to conclude that
\[
\|u\|_{U^\pm_k}\ls{}\|u_0\|_{L^2(\R^3)}+\sup_{h \in H}\Big|\int_{\R^{1+3}}\la f, h\ra_{\C^d} dxdt \Big|,
\]
where $H$ is defined as the set of all $h=\tilde{P}_k h$ such that $\|h\|_{V^\pm_k}=1$.
The claim now follows by using again $\|g\|_{S^{\pm}_k}\ls \|g\|_{V^\pm_k}$.
\end{proof}

\begin{rmk}\label{rmk:u2}
In fact, we have proved a stronger result: In the setting of Lemma \ref{lem:lin}, provided that the right hand side of \eqref{eq:lin} is finite, we can upgrade this estimate to
\[
\|u\|_{U^{\pm}_{k}}\ls{} \|u_0\|_{L^2(\R^3)}+\sup_{g \in G}\Big|\int_{\R^{1+3}}\la f, g\ra_{\C^d} dxdt \Big|.
\]
\end{rmk}

Our resolution space $S^{\pm,\sigma}$ corresponding the Sobolev regularity $\sigma$ --used in Subsection \ref{subsect:proof}-- will be the space of functions in $C(\R,H^\sigma(\R^3;\C^d))$ such that
\[
\| f \|_{S^{\pm,\sigma}} = \| P_{\leq 0} f \|_{S_{\leq 0}^\pm} +
\left( \sum_{k \geq 1} 2^{2\sigma k} \| P_k f \|_{S_k^\pm}^2
\right)^\frac12<+\infty,
\]
which is obviously a Banach space.
\section{Nonlinear estimates and the proof of the main result}\label{sect:nl}

Recall \eqref{DKGf} with the convention $M=m=1$ and use the decomposition $\psi=\Pi_+(D)\psi+\Pi_-(D)\psi$
in the nonlinearity (for all three terms). It then suffices to prove
\begin{align*}
  \Big|\int \la \Pi_{s_2}(D)[\Re\phi \, \beta \Pi_{s_1}(D)\psi_1],\psi_2\ra dxdt\Big|& \ls \|\phi\|_{S^{+,\frac12+\eps}}\|\psi_1\|_{S^{s_1,\eps}}\|\psi_2\|_{S^{s_2,-\eps}}\\
  \Big|\int \la D \ra^{-1}\la \Pi_{s_1}(D)\psi_1, \beta
  \Pi_{s_2}(D)\psi_2\ra \, \overline{\phi} dxdt\Big|& \ls
  \|\phi\|_{S^{+,-\frac12-\eps}}\|\psi_1\|_{S^{s_1,\eps}}\|\psi_2\|_{S^{s_2,\eps}}
\end{align*}
for any choice of signs $s_1,s_2\in \{+,-\}$. By symmetry, this
follows from
\begin{equation}\label{eq:red-est}
  \begin{split}
    &\Big|\int \phi \, \la \Pi_{s_1}(D)\psi_1,\beta\Pi_{s_2}(D)\psi_2\ra dxdt\Big|\\
    \ls{} &
    \|\phi\|_{S^{+,\frac12+\eps_0}}\|\psi_1\|_{S^{s_1,\eps_1}}\|\psi_2\|_{S^{s_2,\eps_2}}
  \end{split}
\end{equation}
where $\eps_0,\eps_1,\eps_2\in \{\pm \eps\}$ such that
$\eps_0+\eps_1+\eps_2=\eps$. More precisely, we will prove this  first on the dyadic level, where all integrals are clearly finite, cp.\ Lemma \ref{lem:lin}.
 
\subsection{Estimates for dyadic pieces}
Our aim will be to identify a function $G: \N^3_0 \rightarrow
(0,\infty)$ such that
\begin{equation} \label{G} \sum_{k,k_1,k_2 \in \N_0 \atop \max(k,k_1,k_2)\sim \med(k,k_1,k_2)}
  \frac{G(k,k_1,k_2)a_{k}
  b_{k_1} c_{k_2} }{2^{\frac{k}2}(\min(k,k_1,k_2)+1)^{10}} \ls \| a \|_{l^2} \| b \|_{l^2} \| c \|_{l^2} 
\end{equation}
for all sequences $a=(a_j)_{j \in \N_0}$ etc.\ in $l^2(\N_0)$. We
write $\mathbf{k}=(k,k_1,k_2)$.

Clearly, \eqref{eq:red-est} is implied by the following key result of
this section:
\begin{pro} \label{PTR} Let $s_1,s_2\in\{+,-\}$. There exists a
  function $G$ satisfying \eqref{G} such that for all
  $\phi=P_{k}\phi,\psi_i=P_{k_i}\Pi_{s_i}(D)\psi_i$, $i=1,2$, the
  following estimate holds true:
  \begin{equation} \label{cunn2} \Big| \int \phi \la \psi_1, \beta
    \psi_2 \ra dx dt\Big| \ls{} G(\mathbf{k}) \| \phi \|_{S^{+}_{k}}
    \| \psi_1 \|_{S^{s_1}_{k_1}} \| \psi_2 \|_{S^{s_2}_{k_2}}.
  \end{equation}
\end{pro}

\begin{proof} We denote the integral on the left hand side of
  \eqref{cunn2} by $I(\mathbf{k})$.
  Without restricting the generality of the argument we can assume
  that $k_1 \leq k_2$.
  We decompose
  \[
  I(\mathbf{k})=I_0(\mathbf{k})+I_1(\mathbf{k})+I_2(\mathbf{k})
  \]
  where
  \begin{align*}
    I_0(\mathbf{k}):=&\sum_{j\in \Z} \int Q_{j}^+ \phi \, \la Q^{s_1}_{\leq j}\psi_1, \beta Q^{s_2}_{\leq j}\psi_2 \ra dx dt\\
    I_1(\mathbf{k}):=&\sum_{j_1\in \Z}\int Q_{<j_1}^+ \phi  \,\la Q^{s_1}_{j_1}\psi_1, \beta Q^{s_2}_{\leq j_1}\psi_2 \ra dxdt\\
    I_2(\mathbf{k}):=&\sum_{j_2\in \Z}\int Q_{<j_2}^+ \phi  \,\la
    Q^{s_1}_{<j_2}\psi_1, \beta Q^{s_2}_{j_2}\psi_2 \ra dx dt
  \end{align*}
 Given the symmetry of the estimate in $k_1$ and $k_2$, we split the argument into two cases.

\medskip

  {\bf Case 1:} $|k-k_2 | \leq 10$.

  {\it Contribution of $I_0(\mathbf{k})$:} We split
  $I_0(\mathbf{k})=I_{01}(\mathbf{k})+I_{02}(\mathbf{k})$ according to
  $j<k_1$ and $j\geq k_1$. Then, due to Lemma \ref{lem:mod} there is
  no contribution if $j<k_1$ in the case $s_1=+,s_2=-$. With all other choices of signs, we estimate
  \[
  I_{01}(\mathbf{k}) \ls \sum_{-k_1\preceq j < k_1} \sum_{n,n'\in
    \Xi_{k_1} \atop |n-n'|\preceq k_1} \|\Gamma_{k_1,n} Q_{j}^+ \phi
  \|_{L^2} \| \la Q_{\leq j}^{s_1} \psi_1, \beta Q_{\leq j}^{s_2}
  \Gamma_{k_1,n'}\psi_2 \ra \|_{L^2},
  \]
  where we used orthogonality, and the non-resonance bound \eqref{eq:non-res} to restrict
  the sum to the range $j\succeq -k_1$.  We conclude from Lemma
  \ref{lem:mod} with $2l=k_1+k_2-k-j$ and Lemma \ref{lem:stable}
  \begin{align*}
    &\| \la Q_{\leq j}^{s_1} \psi_1, \beta Q_{\leq j}^{s_2} \Gamma_{k,n'}\psi_2 \ra \|_{L^2}\\
    \ls{} & 2^{-l} \sum_{\kappa_1,\kappa_2 \in \mathcal{K}_l \atop
      \dist(s_1 \kappa_1,s_2\kappa_2)\ls 2^{-l}}\| Q_{\leq j}^{s_1}
    P_{\kappa_1}\psi_1\|_{L^3_t L^6_x}\|Q_{\leq j}^{s_2}
    P_{\kappa_2}\Gamma_{k_1,n'}\psi_2\|_{L^6_t L^3_x}.
  \end{align*}
  By Part v) of Lemma \ref{lem:stable}, the operators $Q_{\leq j}^{\pm}$  are disposable up to a factor $\la k_1\ra$. Then, we apply Cauchy-Schwarz and perform the cube and
  cap summation and obtain
  \begin{align*}
    I_{01}(\mathbf{k}) & \ls\sum_{-k_1\preceq j < k_1} 2^{-\frac{j}2}
    \| \phi \|_{S_{k}} 2^{-\frac{k_1+k_2-k-j}2} 2^{\frac{2k_1}3} \la k_1\ra \|
    \psi_1 \|_{S_{k_1}^{s_1}}
    2^{\frac{k_1+k_2}6} \la k_1\ra \| \psi_2 \|_{S_{k_2}^{s_2}} \\
    & \ls \la k_1\ra^3 2^{\frac{k_1-k_2}3} 2^{\frac{k}2} \| \phi \|_{S^+_{k}}
    \| \psi_1 \|_{S^{s_1}_{k_1}} \| \psi_2 \|_{S^{s_2}_{k_2}}.
  \end{align*}
  In the range $j \geq k_1$, the operators $Q_{\leq j}^{\pm}$  are disposable and a similar argument above with $l=0$,
  i.e.\ no cap decomposition and no gain from the null-structure,
  gives the bound
  \[
  I_{02}(\mathbf{k}) \ls  2^{\frac{k_1-k_2}3} 
  2^{\frac{k}2} \| \phi \|_{S^+_{k}} \| \psi_1 \|_{S^{s_1}_{k_1}} \|
  \psi_2 \|_{S^{s_2}_{k_2}}.
  \]

{\it Contribution of $I_1(\mathbf{k})$:} We split
  $I_1(\mathbf{k})=I_{11}(\mathbf{k})+I_{12}(\mathbf{k})$ according to
 $j_1 < k_1$ and $j_1\geq k_1$. Again, by Lemma \ref{lem:mod} there is
  no contribution if $j_1<k_1$ in the case $s_1=+,s_2=-$.  With all other choices of signs, we can restrict the sum in $I_{11}$ to $j_1\succeq -k_1$, so that by Lemma \ref{lem:mod} with $2l=k_1+k_2-k-j_1\sim k_1-j_1$ we have
\begin{align*}
I_{11}(\mathbf{k})= &\sum_{-k_1\preceq j_1<k_1}\sum_{n,n' \in \Xi_{k_1} \atop |n-n'|\preceq k_1} \sum_{\kappa_1,\kappa_2 \in \mathcal{K}_l \atop
      \dist(s_1 \kappa_1,s_2\kappa_2)\ls 2^{-l}} \Big\{\\
&\int \Gamma_{k_1,n} Q^+_{<j_1} \phi \cdot \la P_{\kappa_1}Q_{j_1}^{s_1} \psi_1, \beta  P_{\kappa_2}\Gamma_{k_1,n'} Q_{\leq j_1}^{s_2}  \psi_2 \ra dx dt\Big\}
  \end{align*}
In view of Lemma \ref{lem:stable}, we decompose
\[
\Pi_{s_i}(D)P_{\kappa_i}=[\Pi_{s_i}(D)-\Pi_{s_i}(2^{k_i}\omega(\kappa_i ))]P_{\kappa_i}+\Pi_{s_i}(2^{k_i}\omega(\kappa_i ))P_{\kappa_i},
\]
and obtain
  \begin{align*}
    & \|\la P_{\kappa_1}Q_{j_1}^{s_1} \psi_1, \beta  P_{\kappa_2}\Gamma_{k_1,n'} Q_{\leq j_1}^{s_2}  \psi_2 \ra\|_{L^{\frac32}_tL^{\frac65}_x}\\
\ls{} & 2^{-l} \|P_{\kappa_1}Q_{j_1}^{s_1} \psi_1\|_{L^2} \| P_{\kappa_2}\Gamma_{k_1,n'} Q_{\leq j_1}^{s_2}  \psi_2\|_{L^6_tL^3_x}.
  \end{align*}
By H\"older's inequality and Cauchy-Schwarz we obtain
  \begin{align*}
  I_{11}(\mathbf{k})  & \ls{} \sum_{-k_1\preceq j_1<k_1} \Big\{ 2^{-\frac{k_1-j_1}2} \| Q^{s_1}_{j_1} \psi_1\|_{L^2} \Big(\sum_{n \in \Xi_{k_1}} \|\Gamma_{k_1,n} Q^+_{<j_1} \phi\|_{L^3_tL^6_x}^2\Big)^{\frac12} \\
& \qquad \qquad \cdot\Big(\sum_{n' \in \Xi_{k_1}} \sum_{\kappa_2 \in \mathcal{K}_l}\| P_{\kappa_2} \Gamma_{k_1,n'} Q_{\leq j_1}^{s_2}  \psi_2\|_{L^6_tL^3_x}^2 \Big)^{\frac12}\Big\}\\
\ls{}& \sum_{-k_1\preceq j_1<k_1} 2^{-\frac{k_1-j_1}2} 2^{-\frac{j_1}{2}} \|\psi_1\|_{S^{s_1}_{k_1}}2^{\frac{k_1+k}{3}}\la k_1\ra\|\phi\|_{S^+_{k}}2^{\frac{k_1+k_2}{6}}\la k_1\ra\|\psi_2\|_{S^{s_2}_{k_2}}\\
    \ls{}& 2^{\frac{k}2} \la k_1\ra^3\| \phi \|_{S^+_{k}} \|\psi_1\|_{S^{s_1}_{k_1}} \| \psi_2 \|_{S^{s_2}_{k_2}},
  \end{align*}
where we have used Lemma \ref{lem:stable} Part v).

  In the range $j_1 \geq k_1$, we forgo the gain from the null-structure in the
  above argument and obtain
\begin{align*}
  I_{12}(\mathbf{k})\ls{}  &  \sum_{j_1\geq k_1}  2^{-\frac{j_1}{2}} \|\psi_1\|_{S^{s_1}_{k_1}}2^{\frac{k_1+k}{3}}\|\phi\|_{S^+_{k}}2^{\frac{k_1+k_2}{6}}\|\psi_2\|_{S^{s_2}_{k_2}}\\
    \ls{}& 2^{\frac{k}2} \| \phi \|_{S^+_{k}} \|\psi_1\|_{S^{s_1}_{k_1}} \| \psi_2 \|_{S^{s_2}_{k_2}}
  \end{align*}
since the operators $Q_{\leq j}^{\pm}$  are disposable.

{\it Contribution of $I_2(\mathbf{k})$:} As above, we split
  $I_2(\mathbf{k})=I_{21}(\mathbf{k})+I_{22}(\mathbf{k})$ according to
 $j_2 < k_1$ and $j_2\geq k_1$. Again, by Lemma \ref{lem:mod} there is
  no contribution if $j_2<k_1$ in the case $s_1=+,s_2=-$, whereas in all other choices of signs, we can restrict the sum in $I_{21}(\mathbf{k})$ to $j_2\succeq -k_1$, so that by Lemma \ref{lem:mod} with $2l=k_1+k_2-k-j_2\sim k_1-j_2$ we repeat the argument for $ I_{11}(\mathbf{k})$ to obtain
\begin{align*}
  I_{21}(\mathbf{k})  & \ls{}
\sum_{-k_1\preceq j_2<k_1} 2^{\frac{k_1+k}{3}}\la k_1\ra\|\phi\|_{S^+_{k}} 2^{-\frac{k_1-j_2}2}  2^{\frac{k_1}{3}}\la k_1\ra \|\psi_1\|_{S^{s_1}_{k_1}} 2^{-\frac{j_2}{2}}\|\psi_2\|_{S^{s_2}_{k_2}}\\
\ls{}& 2^{\frac{k}2} 2^{\frac{k_1-k}{6}} \la k_1\ra^3\| \phi \|_{S^+_{k}} \|\psi_1\|_{S^{s_1}_{k_1}} \| \psi_2 \|_{S^{s_2}_{k_2}}
\end{align*}
For th range $j_2 \geq k_1$, then the same argument as above, but with no gain from the null-structure, gives the bound
\begin{align*}
  I_{22}(\mathbf{k})\ls{}  &  \sum_{j_2\geq k_1}   2^{\frac{k_1+k}{3}}\|\phi\|_{S^+_{k}}  2^{\frac{k_1}{3}} \|\psi_1\|_{S^{s_1}_{k_1}} 2^{-\frac{j_2}{2}}\|\psi_2\|_{S^{s_2}_{k_2}}\\
\ls{}  &2^{\frac{k}2} 2^{\frac{k_1-k}{6}} \| \phi \|_{S^+_{k}} \|\psi_1\|_{S^{s_1}_{k_1}} \| \psi_2 \|_{S^{s_2}_{k_2}}.
\end{align*}

\medskip

  {\bf Case 2:} $|k_1-k_2 | \leq 10$.

  {\it Contribution of $I_0(\mathbf{k})$:} We split
  $I_0(\mathbf{k})=I_{01}(\mathbf{k})+I_{02}(\mathbf{k})$ according to
  $j<k$ and $j\geq k$. Then, due to Lemma \ref{lem:mod} there is
  no contribution if $j<k$ in the case $s_1=+,s_2=-$ or $s_1=-, s_2=+$ and $k \prec \min(k_1,k_2)$.
In all remaining cases, we can restrict the sum in $I_{01}$ to $j\succeq -k$, so that
  \[
  I_{01}(\mathbf{k}) \ls \sum_{-k\preceq j < k} \sum_{n,n'\in
    \Xi_{k} \atop |n-n'|\preceq k} \| Q_{j}^+ \phi
  \|_{L^2} \| \la Q_{\leq j}^{s_1}\Gamma_{k,n} \psi_1, \beta Q_{\leq j}^{s_2}
  \Gamma_{k,n'}\psi_2 \ra \|_{L^2}.
  \]
 We conclude from Lemma \ref{lem:mod} with $2l=k_1+k_2-k-j$ and Lemma \ref{lem:stable} that
  \begin{align*}
    &\| \la Q_{\leq j}^{s_1}\Gamma_{k,n} \psi_1, \beta Q_{\leq j}^{s_2} \Gamma_{k,n'}\psi_2 \ra \|_{L^2}\\
    \ls{} & 2^{-l} \sum_{\kappa_1,\kappa_2 \in \mathcal{K}_l \atop
      \dist(s_1 \kappa_1,s_2\kappa_2)\ls 2^{-l}}\| Q_{\leq j}^{s_1}
    P_{\kappa_1}\Gamma_{k,n}\psi_1\|_{L^3_t L^6_x}\|Q_{\leq j}^{s_2}
    P_{\kappa_2}\Gamma_{k,n'}\psi_2\|_{L^6_t L^3_x}.
  \end{align*}
  By Part v) of Lemma \ref{lem:stable}, the operators $Q_{\leq j}^{\pm}$ are disposable up to a factor $\la k\ra$. Then, we apply Cauchy-Schwarz and perform the cube and
  cap summation and obtain
  \begin{align*}
    I_{01}(\mathbf{k}) & \ls\sum_{-k\preceq j < k} 2^{-\frac{j}2}
    \| \phi \|_{S_{k}} 2^{-\frac{k_1+k_2-k-j}2} 2^{\frac{k+k_1}3}\la k\ra \|
    \psi_1 \|_{S_{k_1}^{s_1}}
    2^{\frac{k+k_2}6} \la k\ra \| \psi_2 \|_{S_{k_2}^{s_2}} \\
    & \ls \la k\ra^3 2^{\frac{k-k_1}2} 2^{\frac{k}2} \| \phi \|_{S^+_{k}}
    \| \psi_1 \|_{S^{s_1}_{k_1}} \| \psi_2 \|_{S^{s_2}_{k_2}}.
  \end{align*}
  Let us now consider the range $j \geq k$. Now, by Part v) of Lemma \ref{lem:stable}, the operators $Q_{\leq j}^{\pm}$ are disposable. In the case $s_1=+$, $s_2=-$ or in the case $s_1=-$, $s_2=+$ and $k \prec \min(k_1,k_2)$, Lemma \ref{lem:mod} implies that there is only a contribution if $j\succeq k_1$. Then, we obtain from the above argument with $l=0$
\begin{align*}
  I_{02}(\mathbf{k}) & \ls  \sum_{j \succeq k_1}2^{-\frac{j}{2}}
  \| \phi \|_{S^+_{k}}  2^{\frac{k+k_1}3} \| \psi_1 \|_{S^{s_1}_{k_1}} 2^{\frac{k+k_2}6}\|
  \psi_2 \|_{S^{s_2}_{k_2}}\\
\ls{}&  2^{\frac{k}2}
  \| \phi \|_{S^+_{k}}  \| \psi_1 \|_{S^{s_1}_{k_1}} \|\psi_2 \|_{S^{s_2}_{k_2}}.
\end{align*}
In the case $s_1=s_2$, \eqref{eq:ft} implies that the integral is nonzero only if the frequencies in the supports of $\widehat{\psi_1}$ and $\widehat{\psi_2}$ make an angle of at most $2^{k-k_1}$, hence, we choose $l=k_1-k$. In the remaining case where $s_1=-$, $s_2=+$ and $k \succeq \min(k_1,k_2)$ we choose $l=0$.
Again, arguing as for $I_{01}(\mathbf{k})$ we obtain
\begin{align*}
  I_{02}(\mathbf{k})& \ls \sum_{ j \geq k} 2^{-\frac{j}2}
    \| \phi \|_{S_{k}} 2^{-l} 2^{\frac{k+k_1}3} \|
    \psi_1 \|_{S_{k_1}^{s_1}}
    2^{\frac{k+k_2}6}  \| \psi_2 \|_{S_{k_2}^{s_2}} \\
& \ls 2^{\frac{k}2}
  \| \phi \|_{S^+_{k}}  \| \psi_1 \|_{S^{s_1}_{k_1}} \|\psi_2 \|_{S^{s_2}_{k_2}}.
\end{align*}

{\it Contribution of $I_1(\mathbf{k})$:} Again, we split
  $I_1(\mathbf{k})=I_{11}(\mathbf{k})+I_{12}(\mathbf{k})$ according to
  $j_1<k$ and $j_1\geq k$. Then, due to Lemma \ref{lem:mod} there is
  no contribution if $j_1<k$ in the case $s_1=+,s_2=-$ or $s_1=-, s_2=+$ and $k \prec \min(k_1,k_2)$.
In all remaining cases, we can restrict the sum in $I_{11}$ to $j_1\succeq -k$, so that by Lemma \ref{lem:mod} with $2l=k_1+k_2-k-j_1$ we have
\begin{align*}
I_{11}(\mathbf{k})= &\sum_{-k\preceq j_1<k}\sum_{n,n' \in \Xi_{k} \atop |n-n'|\preceq k} \sum_{\kappa_1,\kappa_2 \in \mathcal{K}_l \atop
      \dist(s_1 \kappa_1,s_2\kappa_2)\ls 2^{-l}} \Big\{\\
&\int  Q^+_{<j_1} \phi \cdot \la P_{\kappa_1}Q_{j_1}^{s_1}\Gamma_{k,n} \psi_1, \beta  P_{\kappa_2}\Gamma_{k,n'} Q_{\leq j_1}^{s_2}  \psi_2 \ra dx dt\Big\}
  \end{align*}
Using Lemma \ref{lem:stable}, we obtain
  \begin{align*}
    & \|\la P_{\kappa_1}Q_{j_1}^{s_1} \Gamma_{k,n} \psi_1, \beta  P_{\kappa_2}\Gamma_{k,n'} Q_{\leq j_1}^{s_2}  \psi_2 \ra\|_{L^{\frac32}_tL^{\frac65}_x}\\
\ls{} & 2^{-l} \|P_{\kappa_1}Q_{j_1}^{s_1} \Gamma_{k,n} \psi_1\|_{L^2} \| P_{\kappa_2}\Gamma_{k,n'} Q_{\leq j_1}^{s_2}  \psi_2\|_{L^6_tL^3_x}.
  \end{align*}
By H\"older's inequality and Cauchy-Schwarz we obtain
  \begin{align*}
  I_{11}(\mathbf{k})  & \ls{} \sum_{-k\preceq j_1<k} \|Q^+_{<j_1} \phi\|_{L^3_tL^6_x} 2^{-\frac{k_1+k_2-k-j_1}2} \| Q^{s_1}_{j_1} \psi_1\|_{L^2}  \|  Q_{\leq j_1}^{s_2}  \psi_2\|_{L^6_tL^3_x[k_2;l,k]} \\
\ls{}& \sum_{-k\preceq j_1<k} 2^{\frac{2k}{3}}\la k\ra\|\phi\|_{S^+_{k}} 2^{-\frac{k_1+k_2-k-j_1}2}  2^{-\frac{j_1}{2}}\|\psi_1\|_{S^{s_1}_{k_1}}2^{\frac{k+k_2}{6}}\la k\ra\|\psi_2\|_{S^{s_2}_{k_2}}\\
    \ls{}& 2^{\frac{k}2} 2^{\frac{5}{6}(k-k_1)}\la k\ra^3\| \phi \|_{S^+_{k}} \|\psi_1\|_{S^{s_1}_{k_1}} \| \psi_2 \|_{S^{s_2}_{k_2}},
  \end{align*}
where we have also used Lemma \ref{lem:stable} Part v).

 Let us now consider the case $j_1 \geq k$. We use a similar dichotomy as for $I_{02}(\mathbf{k})$. In the case $s_1=+$, $s_2=-$ or in the case $s_1=-$, $s_2=+$ and $k \prec \min(k_1,k_2)$, Lemma \ref{lem:mod} implies that there is only a contribution if $j_1\succeq k_2$.
In that case, we obtain from the above argument with $l=0$
\begin{align*}
  I_{12}(\mathbf{k})\ls{}  &  \sum_{j_1\succeq k_2}  2^{\frac{2k}{3}}\|\phi\|_{S^+_{k}}  2^{-\frac{j_1}{2}}\|\psi_1\|_{S^{s_1}_{k_1}}2^{\frac{k+k_2}{6}}\|\psi_2\|_{S^{s_2}_{k_2}}\\
\ls{}  & 2^{\frac{k}2} 2^{\frac{1}{3}(k-k_2)}\| \phi \|_{S^+_{k}} \|\psi_1\|_{S^{s_1}_{k_1}} \| \psi_2 \|_{S^{s_2}_{k_2}}.
  \end{align*}

In the case $s_1=s_2$, \eqref{eq:ft} implies that the integral is nonzero only if the frequencies in the supports of $\widehat{\psi_1}$ and $\widehat{\psi_2}$ make an angle of at most $2^{k-k_1}$, hence, we choose $l=k_1-k$. In the remaining case where $s_1=-$, $s_2=+$ and $k \succeq \min(k_1,k_2)$ we choose $l=0$.
By the argument above we obtain
\begin{align*}
  I_{12}(\mathbf{k})\ls{}  &  \sum_{j_1\geq k}  2^{\frac{2k}{3}} \|\phi\|_{S^+_{k}} 2^{-l}  2^{-\frac{j_1}{2}}\|\psi_1\|_{S^{s_1}_{k_1}}2^{\frac{k+k_2}{6}}\|\psi_2\|_{S^{s_2}_{k_2}}\\
\ls{}  & 2^{\frac{k}2} 2^{\frac{5}{6}(k-k_1)}\| \phi \|_{S^+_{k}} \|\psi_1\|_{S^{s_1}_{k_1}} \| \psi_2 \|_{S^{s_2}_{k_2}}.
  \end{align*}

{\it Contribution of $I_2(\mathbf{k})$:} This is treated in the same way as $I_1(\mathbf{k})$.
\end{proof}

\begin{rmk}\label{rmk:ang-reg}
  Using $V^2$-based spaces one can avoid the logarithmic divergencies in Part v) of Lemma \ref{lem:stable}. We expect that one would obtain a result in the critical Besov space $\dot B^{0,\eps}_{2,1}\times\dot B^{\frac12,\eps}_{2,1}\times \dot B^{-\frac12,\eps}_{2,1}$, where $\eps>0$ accounts for a bit of
  angular regularity (somewhat strengthening the null-structure and this way eliminating any logarithmic factors). This would improve the result in
  \cite{Wa} (which corresponds to $\eps=1$) in the massive case, however, we will not pursue these
  matters here.
\end{rmk}
\subsection{Proof of Theorem \ref{thm:main}}\label{subsect:proof}
Again, for notational convenience, let $m=M=1$. Fix $\eps>0$. We will construct a solution \[(\psi_+,\psi_-,\phi_+)\in \mathbf{S}^\eps:=S^{+,\eps}\times S^{-,\eps}\times S^{+,\frac{1}{2}+\eps}\] of the system \eqref{DKGf} in integral form, i.e.\
\begin{align*}
\psi_+(t)=&e^{- it \la D \ra}\Pi_+(D)\psi_0 +i\int_0^t e^{-i(t-s) \la D \ra}\Pi_+(D)[\Re\phi_+\beta (\psi_++\psi_-)]ds\\
\psi_-(t)=&e^{it \la D \ra}\Pi_-(D)\psi_0 +i\int_0^t e^{i(t-s) \la D \ra}\Pi_-(D)[\Re\phi_+\beta (\psi_++\psi_-)]ds\\
\phi_+(t)=&e^{- it \la D \ra}\phi_{+,0} +i\int_0^t e^{- i(t-s) \la D \ra}\la D\ra^{-1}\la (\psi_++\psi_-),\beta (\psi_++\psi_-)\ra ds,
\end{align*}
provided that the initial data satisfy
\[\|\psi_{0}\|_{H^\eps(\R^3)}\leq \delta, \quad \|\phi_{+,0}\|_{H^{\frac12+\eps}(\R^3)}\leq \delta,\]
for sufficiently small $\delta>0$. Let $T(\psi_+,\psi_-,\phi_+)$ denote the operator defined by the right hand side of the above formula.

By the results of the previous subsection and Lemma \ref{lem:lin} we conclude
\begin{align*}
&\|T(\psi_+,\psi_-,\phi_+)\|_{\mathbf{S}^\eps}\\
\ls{} &\delta +\|\phi_+\|_{S^{+,\frac{1}{2}+\eps}}(\|\psi_+\|_{S^{+,\eps}}+\|\psi_-\|_{S^{-,\eps}}) +(\|\psi_+\|_{S^{+,\eps}}+\|\psi_-\|_{S^{-,\eps}})^2\\
\ls{} &\delta+\|(\psi_+,\psi_-,\phi_+)\|_{\mathbf{S}^\eps}^2,
\end{align*}
and similar estimates for differences. Hence, in a small closed ball in the complete space $\mathbf{S}^\eps$ we can invoke the contraction mapping principle to obtain a unique solution. Further, continuous dependence on the initial data is an easy consequence.

It remains to prove that these solutions scatter, which we will only do for $t \to +\infty$, the other case being similar.
It suffices to show that for a solution $(\psi_+,\psi_-,\phi_+)\in \mathbf{S}^\eps$ we have convergence of the integrals, i.e.\
\begin{align*}
&\lim_{t \to \infty} \int_0^t e^{- i(-s) \la D \ra}\Pi_+(D)[\Re\phi_+\beta (\psi_++\psi_-)]ds \in H^\eps(\R^3), \\
&\lim_{t \to \infty} \int_0^t e^{i(-s) \la D \ra}\Pi_-(D)[\Re\phi_+\beta (\psi_++\psi_-)]ds \in H^\eps(\R^3),\\
&\lim_{t \to \infty} \int_0^t e^{- i(-s) \la D \ra}\la D\ra^{-1}\la (\psi_++\psi_-),\beta (\psi_++\psi_-)\ra ds \in H^{\frac12+\eps}(\R^3).
\end{align*}
We simply observe that this is a by-product of the linear theory provided by Lemma \ref{lem:lin}.
Indeed, by Remark \ref{rmk:u2} it follows that on the dyadic level these integrals are in fact in $U^{\pm}_k$ and this is square-summable. From this it follows that they are in the space
\[ V^2(\R;H^\eps(\R^3)) \times  V^2(\R;H^\eps(\R^3))
\times  V^2(\R;H^{\frac12+\eps}(\R^3)). 
\]
Functions of bounded $2-$variation have limits at infinity \cite[Prop.~2.2]{HHK} which proves the scattering claim.

\subsection*{Acknowledgement}
The authors thank Sigmund Selberg and Achenef Tesfahun for spotting a couple of typos and flaws in a previous version of the paper.

The first author was supported in part by NSF grant DMS-1001676.  The
second author acknowledges support from the German Research
Foundation, Collaborative Research Center 701. Part of this research
has been carried out while both authors participated in the Trimester
Program \emph{Harmonic Analysis and Partial Differential Equations} at
the Hausdorff Research Institute for Mathematics in Bonn.

\bibliographystyle{plain} \bibliography{dkg-refs}

\begin{thebibliography}{10}

\bibitem{Ba}
Alain Bachelot.
\newblock Probl\`eme de {C}auchy global pour des syst\`emes de
  {D}irac-{K}lein-{G}ordon.
\newblock {\em Ann. Inst. H. Poincar\'e Phys. Th\'eor.}, 48(4):387--422, 1988.

\bibitem{BeBe}
Michael Beals and Max B{\'e}zard.
\newblock Low regularity local solutions for field equations.
\newblock {\em Comm. Partial Differential Equations}, 21(1-2):79--124, 1996.

\bibitem{BH}
Ioan Bejenaru and Sebastian Herr.
\newblock The cubic {D}irac equation: {S}mall initial data in
  {$H^1(\mathbb{R}^3)$}.
\newblock {\em Comm. Math. Phys. (online first)}, 2014.

\bibitem{BjDr}
James~D. Bjorken and Sidney~D. Drell.
\newblock {\em Relativistic quantum mechanics}.
\newblock McGraw-Hill Book Co., New York-Toronto-London, 1964.

\bibitem{Bo}
Nikolaos Bournaveas.
\newblock Local existence of energy class solutions for the
  {D}irac-{K}lein-{G}ordon equations.
\newblock {\em Comm. Partial Differential Equations}, 24(7-8):1167--1193, 1999.

\bibitem{ChGl}
John~M. Chadam and Robert~T. Glassey.
\newblock On certain global solutions of the {C}auchy problem for the
  (classical) coupled {K}lein-{G}ordon-{D}irac equations in one and three space
  dimensions.
\newblock {\em Arch. Rational Mech. Anal.}, 54:223--237, 1974.

\bibitem{CK}
Michael Christ and Alexander Kiselev.
\newblock Maximal functions associated to filtrations.
\newblock {\em J. Funct. Anal.}, 179(2):409--425, 2001.

\bibitem{DaFoSe}
Piero D'Ancona, Damiano Foschi, and Sigmund Selberg.
\newblock Null structure and almost optimal local regularity for the
  {D}irac-{K}lein-{G}ordon system.
\newblock {\em J. Eur. Math. Soc. (JEMS)}, 9(4):877--899, 2007.

\bibitem{DeFa}
Jean-Marc Delort and Daoyuan Fang.
\newblock Almost global existence for solutions of semilinear {K}lein-{G}ordon
  equations with small weakly decaying {C}auchy data.
\newblock {\em Comm. Partial Differential Equations}, 25(11-12):2119--2169,
  2000.

\bibitem{Ge}
Pierre Germain.
\newblock Global existence for coupled {K}lein-{G}ordon equations with
  different speeds.
\newblock {\em Ann. Inst. Fourier (Grenoble)}, 61(6):2463--2506 (2012), 2011.

\bibitem{HHK}
Martin Hadac, Sebastian Herr, and Herbert Koch.
\newblock Well-posedness and scattering for the {KP}-{II} equation in a
  critical space.
\newblock {\em Ann. Inst. H. Poincar\'e Anal. Non Lin\'eaire}, 26(3):917--941,
  2009.

\bibitem{HTT}
Sebastian Herr, Daniel Tataru, and Nikolay Tzvetkov.
\newblock Global well-posedness of the energy-critical nonlinear
  {S}chr\"odinger equation with small initial data in {$H^1(\Bbb T^3)$}.
\newblock {\em Duke Math. J.}, 159(2):329--349, 2011.

\bibitem{Kl}
Sergiu Klainerman.
\newblock Global existence of small amplitude solutions to nonlinear
  {K}lein-{G}ordon equations in four space-time dimensions.
\newblock {\em Comm. Pure Appl. Math.}, 38(5):631--641, 1985.

\bibitem{KT99}
Sergiu Klainerman and Daniel Tataru.
\newblock On the optimal local regularity for {Y}ang-{M}ills equations in
  {${\bf R}^{4+1}$}.
\newblock {\em J. Amer. Math. Soc.}, 12(1):93--116, 1999.

\bibitem{KT}
Herbert Koch and Daniel Tataru.
\newblock Dispersive estimates for principally normal pseudodifferential
  operators.
\newblock {\em Comm. Pure Appl. Math.}, 58(2):217--284, 2005.

\bibitem{KVT}
Herbert Koch, Daniel Tataru, and Monica Visan.
\newblock {\em Dispersive Equations and Nonlinear Waves}, volume~45 of {\em
  Oberwolfach Seminars}.
\newblock Springer Basel, 2014.

\bibitem{S12}
Tobias {Schottdorf}.
\newblock {Global existence without decay for quadratic Klein-Gordon
  equations}.
\newblock arXiv:1209.1518 [math.AP].

\bibitem{S08}
Sigmund Selberg.
\newblock Anisotropic bilinear {$L^2$} estimates related to the 3{D} wave
  equation.
\newblock {\em Int. Math. Res. Not. IMRN}, pages Art. ID rnn 107, 63, 2008.

\bibitem{St}
Jacob Sterbenz.
\newblock Angular regularity and {S}trichartz estimates for the wave equation.
\newblock {\em Int. Math. Res. Not.}, (4):187--231, 2005.
\newblock With an appendix by Igor Rodnianski.

\bibitem{Wa}
Xuecheng Wang.
\newblock On global existence of 3d charge critical {D}irac-{K}lein-{G}ordon
  system.
\newblock arXiv:1311.6068 [mathAP].

\end{thebibliography}

\end{document}